\documentclass[12pt]{amsart}

\newtheorem{theorem}{Theorem}[section]

\newtheorem{corollary}[theorem]{Corollary}

\newtheorem{proposition}[theorem]{Proposition}

\theoremstyle{definition}

\newtheorem{example}[theorem]{Example}

\theoremstyle{parrafo}

\begin{document}

\title[]{Comparison of differences between arithmetic and geometric means}

\author{J. M. Aldaz}
\address{Departamento de Matem\'aticas,
Universidad  Aut\'onoma de Madrid, Cantoblanco 28049, Madrid, Spain.}
\email{jesus.munarriz@uam.es}


\thanks{Partially supported by Grant  MTM2009-12740-C03-03 of the
D.G.I. of Spain}






\begin{abstract} We complement a recent result of S. Furuichi, by showing
that the differences $\sum_{i=1}^n \alpha_i x_i - \prod_{i=1}^n x_i^{\alpha_i}$ associated
to distinct  sequences of weights are comparable, with constants that
depend on the smallest and largest quotients of the weights. 
\end{abstract}


\maketitle


\markboth{J. M. Aldaz}{AM-GM}

\section{Introduction} It is well known that  the inequality between
arithmetic and geometric means is self-improving,  that is, increasingly better versions of the AM-GM inequality
can be obtained simply by repeated applications of itself.
The simplest case of the inequality, $\sqrt{ x y} \le (x + y)/2$,  is trivial and equivalent to
$(\sqrt{ x } - \sqrt{y})^2 \ge 0$. By repeated self-improvement 
 (cf. \cite{St} for more details)
 it is possible to
obtain from $(\sqrt{ x} - \sqrt{ y})^2 \ge 0$  the  
general
AM-GM inequality
\begin{equation}\label{AMG}
\prod_{i=1}^n x_i^{\alpha_i}  
\le 
\sum_{i=1}^n \alpha_i x_i.
\end{equation}
Of course,  there
are more efficient ways to prove (\ref{AMG}). 
Nevertheless, the idea of self-improvement is a useful one:
It has recently been utilized 
 to find refinements of the
AM-GM inequality, and to give better proofs of existing refinements.
For instance,  self-improvement via the change of variables $x_i = y_i^s$,
 was used in \cite{A1}, with $s=1/2$, to show that 
 \begin{equation}\label{refAMGM} 
\sum_{i=1}^n \alpha_i x_i - \prod_{i=1}^n x_i^{\alpha_i}
\ge
\sum_{i=1}^n \alpha_i \left(x_i^{1/2}- \sum_{k=1}^n \alpha_k x_k^{1/2}\right)^2.
\end{equation}
 Observe that the right hand side of (\ref{refAMGM}) is the variance 
$\operatorname{Var}(x^{1/2})$ of
the vector $x^{1/2} = (x_1^{1/2},\dots,x_n^{1/2})$ with respect to the discrete probability $\sum_{i=1}^n \alpha_i \delta_{x_i}$. While a large variance
(of $x^{1/2}$) pushes the arithmetic and geometric means apart, no conclusions
can be derived from a small
variance, as noted in \cite{A1}. Suppose $n >> 1$ and $x_1 = \cdots =x_n$, so $\operatorname{Var}(x^{1/2})= 0$. If $\alpha_j$ is the smallest weight, letting $x_j\downarrow 0$ leaves $\operatorname{Var}(x^{1/2})$
and $\sum_{i=1}^n \alpha_i x_i$  
essentially unchanged,  while $\prod_{i=1}^n x_i^{\alpha_i}$ drops to 0. Thus, an upper
bound for $\sum_{i=1}^n \alpha_i x_i - \prod_{i=1}^n x_i^{\alpha_i}$ in terms of 
$\operatorname{Var}(x^{1/2})$ alone cannot be found, so it is natural to seek an alternative
way to control this difference.

Using an idea originally due to S. Dragomir (cf. \cite{Dra}), 
S. Furuichi  showed by self-improvement that 
$\sum_{i=1}^n \alpha_i x_i - \prod_{i=1}^n x_i^{\alpha_i} \ge 
n \alpha_{\min} \left( n^{-1}\sum_{i=1}^n  x_i - \prod_{i=1}^n x_i^{1/n}\right)$, where $\alpha_{\min}:= \min\{\alpha_1, \dots, \alpha_n\}$, cf. \cite{Fu}. This result generalizes one side of the 
two-sided refinement of
Young's inequality presented in \cite[Lemma 2.1]{A4}. 
We
show in this note that
 the equal weights AM-GM difference can be utilized  to give an entirely analogous
upper bound, with $\alpha_{\max}:= \max\{\alpha_1, \dots, \alpha_n\}$ replacing $\alpha_{\min}$.
More generally, we shall see that differences between arithmetic and geometric means associated
to different sequences of weights are comparable, with constants depending on the maxima and minima of the
sequence of quotients of weights. A standard application to H\"older's
inequality is  presented next. We finish with a discussion regarding
the ``typical size", in a certain probabilistic sense, of 
$\sum_{i=1}^n \alpha_i x_i - \prod_{i=1}^n x_i^{\alpha_i}$ and of 
$\prod_{i=1}^n x_i^{\alpha_i}/\sum_{i=1}^n \alpha_i x_i$ 
when $n >>1$.

\section{Self-bounds on AM-GM differences}

The first version of this note was written without the author being
aware of \cite{Dra}.
 Originally, Theorem \ref{AMGMrFu} had  weaker bounds, given in terms of
$\alpha_{\min}/\beta_{\max}$ and $\alpha_{\max}/\beta_{\min}$,
where 
$\alpha_{\min}:= \min\{\alpha_1, \dots, \alpha_n\}$, $\alpha_{\max}:= \max\{\alpha_1, \dots, \alpha_n\}$, and analogously for 
$\beta_{\min}$ and $\beta_{\max}$.
It was pointed out by an anonymous referee that the stronger 
bounds appearing in  (\ref{refAMGMFugen}) could be obtained by applying 
\cite[Theorem 1]{Dra} (which refines Jensen's inequality) to the
function $f(t) = e^t$. This also has been observed by 
Flavia Corina Mitroi (personal communication, cf.
\cite{Fla}). 

Here we note that the same
self-improvement argument used by S. Dragomir in \cite{Dra}, yields (\ref{refAMGMFugen}) directly from the AM-GM
inequality. We also study the equality conditions,  which are not
considered in \cite{Dra}.

\begin{theorem}\label{AMGMrFu}  For $n\ge 2$ and $i=1,\dots, n$, let $x_i\ge 0$, 
 and let
$\alpha_i, \beta_i > 0$ satisfy $\sum_{i=1}^n \alpha_i = \sum_{i=1}^n \beta_i = 1$.  Then we have 
\begin{equation}\label{refAMGMFugen} 
\min_{k=1, \dots, n}\left\{\frac{\alpha_{k}}{\beta_{k}}\right\}\left(\sum_{i=1}^n \beta_i x_i - \prod_{i=1}^n x_i^{\beta_i}
\right) 
\le
\sum_{i=1}^n \alpha_i x_i - \prod_{i=1}^n x_i^{\alpha_i}
\le
\max_{k=1, \dots, n}\left\{\frac{\alpha_{k}}{\beta_{k}}\right\}\left(\sum_{i=1}^n \beta_i x_i - \prod_{i=1}^n x_i^{\beta_i}\right).
\end{equation}
Regarding the equality conditions,  
set $A:=\{i: \alpha_i/\beta_i = \min\{\alpha_k/\beta_k: 1\le k\le n\}, i=1, \dots, n\}$ and   
$B=\{i: \alpha_i/\beta_i = \max\{\alpha_k/\beta_k: 1\le k\le n\},
i=1, \dots, n\}$. Then we have equality on the left hand side of (\ref{refAMGMFugen}) if and only if 
for every $j\in \{1, \cdots, n\} \setminus A$, 
\begin{equation}\label{eqleft} 
x_j =  \left(\prod_{i\in A} x_i^{\alpha_i}\right)^{\frac{1}{\sum_{i\in A} \alpha_i}}
= \prod_{i=1}^n x_i^{\alpha_i}, \mbox{ \ or equivalently, \ }
x_j =  \left(\prod_{i\in A} x_i^{\beta_i}\right)^{\frac{1}{\sum_{i\in A} \beta_i}}= \prod_{i=1}^n x_i^{\beta_i},
\end{equation}
while equality holds on the right  hand side of (\ref{refAMGMFugen}) if and only if 
for every $j\in \{1, \cdots, n\} \setminus B$, 
\begin{equation}\label{eqright} 
x_j =  \left(\prod_{i\in B} x_i^{\alpha_i}\right)^{\frac{1}{\sum_{i\in B} \alpha_i}}
= \prod_{i=1}^n x_i^{\alpha_i}, \mbox{ \  or equivalently, \ }
x_j =  \left(\prod_{i\in B} x_i^{\beta_i}\right)^{\frac{1}{\sum_{i\in B} \beta_i}}
= \prod_{i=1}^n x_i^{\beta_i}.
\end{equation}
\end{theorem}

Let $|A|$ and $|B|$ denote the cardinalities of the sets $A$ and $B$.
Observe that when  $|A|=1$,  equality holds on the
left hand side of (\ref{refAMGMFugen}) 
 if and only if $x_1 = \cdots =x_n$, 
and likewise for the right hand side when $|B|=1$.

Before proving Theorem \ref{AMGMrFu}, we present an example illustrating the need  for the
slightly complicated formulation of the equality conditions 
(which were incorrectly stated in the first version of this
note, after reading \cite{Fu} without enough care). 
We want to stress the fact that the left and right hand sides of (\ref{refAMGMFugen})
have to be dealt with separately.

\begin{example} In formula (\ref{refAMGMFugen}), let $n=3$, $\alpha_1 = 2/3$, $\alpha_2 = \alpha_3 = 1/6$,
and for $i= 1, 2, 3$, 
$\beta_i = 1/3$. As a normalization, suppose also that $(xyz)^{1/3} =1$. Then $|B|=1$, and equality
holds on the right hand side of (\ref{refAMGMFugen}) if and only if $x=y=z=1$. However, this
condition is too strong for the left hand side since  $1 < |A|=2$; there, equality holds if and only if $x= yz=1$.
\end{example}

\begin{proof} The second inequality in  (\ref{refAMGMFugen}) is equivalent
to
\begin{equation}\label{step1} 
\prod_{i=1}^n x_i^{\beta_i}
\le
\sum_{i=1}^n \left(\beta_{i} - \min_{k=1, \dots, n}\left\{\frac{\beta_{k}}{\alpha_{k}}\right\} \alpha_i
\right) x_i  +  
\min_{k=1, \dots, n}\left\{\frac{\beta_{k}}{\alpha_{k}}\right\} \prod_{i=1}^n x_i^{\alpha_i}.
\end{equation}
Writing 
\begin{equation}\label{step2} 
\prod_{i=1}^n x_i^{\beta_i}
= 
\prod_{i=1}^n x_i^{\beta_{i} - \min_{k=1, \dots, n}\left\{\frac{\beta_{k}}{\alpha_{k}}\right\} \alpha_i}
 \prod_{i=1}^n x_i^{\min_{k=1, \dots, n}\left\{\frac{\beta_{k}}{\alpha_{k}}\right\}\alpha_i},
\end{equation}
we see that (\ref{step1}) is just the  AM-GM inequality,  
 since $0 = \beta_{i} - 
\frac{\beta_{i}}{\alpha_{i}} \alpha_{i} \le
\beta_{i} - \min_{k=1, \dots, n} 
\left\{\frac{\beta_{k}}{\alpha_{k}}\right\} \alpha_{i}$
and
$$
\sum_{i=1}^n \left(\beta_{i} - \min_{k=1, \dots, n} 
\left\{\frac{\beta_{k}}{\alpha_{k}}\right\} \alpha_i
\right)   +  
\min_{k=1, \dots, n} 
\left\{\frac{\beta_{k}}{\alpha_{k}}\right\} =1.$$
 To obtain the first inequality 
in  formula (\ref{refAMGMFugen}), 
multiply both sides of the second inequality by $\min_{k=1, \dots, n}\left\{\frac{\beta_{k}}{\alpha_{k}}\right\}$,
and note that this is just the first inequality with the roles of
the $\alpha$'s and the $\beta$'s interchanged. Alternatively, it is
possible to prove the first inequality directly, using the
same argument as above, and derive the second from the first.

For the case of equality, set $q_i := \alpha_i/\beta_i$ and
 suppose that $q_1\le \cdots \le q_n$ (by
rearranging the sequences of weights, if needed). If $r:= |A| =n$, 
then also $s:=|B| =n$, all $q_i =1$, and thus, for $i=1,\dots, n$, $\alpha_i =\beta_i$, 
whence equality trivially holds on both sides of (\ref{refAMGMFugen}),
without any conditions imposed on  the variables $x_i$. 
If $r < n$, then $r < n - s +1\le n$, and 
$q_{n-s} < q_{n-s + 1} =\cdots = q_{n}$. Assume that equality holds
on the right hand side of (\ref{refAMGMFugen}), or, equivalently, that
\begin{equation}\label{step1eq} 
\prod_{i=1}^n x_i^{\beta_i}
=
\sum_{i=1}^n \left(\beta_{i} - \min_{k=1, \dots, n}\left\{\frac{\beta_{k}}{\alpha_{k}}\right\} \alpha_i
\right) x_i  +  
\min_{k=1, \dots, n}\left\{\frac{\beta_{k}}{\alpha_{k}}\right\} \prod_{i=1}^n x_i^{\alpha_i}.
\end{equation}
Removing the zero terms from the preceding sum, we see that (\ref{step1eq}) holds if and only if  
\begin{equation}\label{step2eq} 
\sum_{i=1}^{n - s} \left(\beta_{i} - \min_{k=1, \dots, n}\left\{\frac{\beta_{k}}{\alpha_{k}}\right\} \alpha_i
\right) x_i  +  
\min_{k=1, \dots, n}\left\{\frac{\beta_{k}}{\alpha_{k}}\right\} \prod_{i=1}^n x_i^{\alpha_i}
=
\prod_{i=1}^n x_i^{\beta_i}
\end{equation}
It now follows from the equality condition in the  AMGM inequality  that (\ref{step2eq}) holds if and only if
\begin{equation}\label{step4eq}
c:= x_1 =\cdots = x_{n-s} = 
 \prod_{i=1}^n x_i^{\alpha_i}
 =
\prod_{i=1}^n x_i^{\beta_i}.
\end{equation}

Thus,
\begin{equation}\label{step5eq}
c= c^{\sum_{i=1}^{n - s}  \alpha_i}
 \prod_{i= n -s + 1}^n x_i^{\alpha_i}
 =
c^{\sum_{i=1}^{n - s}  \beta_i} \prod_{i= n -s + 1}^n x_i^{\beta_i},
\end{equation}
and we obtain the conditions appearing in (\ref{eqleft}) by
simplifying and solving for $c$. 
To see that one of these conditions is redundant, recall that for $i = n -s + 1,\dots, n$,
$\beta_i = \min_{k=1, \dots, n}\left\{\frac{\beta_{k}}{\alpha_{k}}\right\}\alpha_i$,
so 
$$\left(\prod_{i=n -s + 1}^n x_i^{\alpha_i}\right)^{\frac{1}{\sum_{i=n -s + 1}^n \alpha_i}}
= 
\left(\prod_{i=n -s + 1}^n x_i^{\beta_i}\right)^{\frac{1}{\sum_{i=n -s + 1}^n \beta_i}}.
$$
The equality conditions for the left hand side of (\ref{refAMGMFugen}) can be obtained in the same way,
or, alternatively,
from those for the right hand side, by interchanging the roles of the $\alpha$'s and
$\beta$'s.
\end{proof}

Recall that $\alpha_{\min}:= \min\{\alpha_1, \dots, \alpha_n\}$ and $\alpha_{\max}:= \max\{\alpha_1, \dots, \alpha_n\}$.

\begin{corollary}\label{cor1}  Under the same hypotheses and with the notation of
Theorem  \ref{AMGMrFu}, let $\beta_i = 1/n$
 for all $i = 1, \dots, n$. Then
\begin{equation}\label{refAMGMFeqwe} 
n \alpha_{\min}\left(\frac1n \sum_{i=1}^n  x_i - \prod_{i=1}^n x_i^{1/n}\right) 
\le
\sum_{i=1}^n \alpha_i x_i - \prod_{i=1}^n x_i^{\alpha_i}
\le
n \alpha_{\max}\left(\frac1n \sum_{i=1}^n  x_i - \prod_{i=1}^n x_i^{1/n}\right).
\end{equation}
\end{corollary}

The left hand side of (\ref{refAMGMFeqwe}) is 
essentially the content of \cite{Fu}. For completeness, next
we consider the equality case in the Dragomir-Jensen inequality for strictly
convex functions.

\begin{proposition}\label{DraJen} For $n\ge 2$ and $i=1,\dots, n$, 
 let
$\alpha_i, \beta_i > 0$ satisfy $\sum_{i=1}^n \alpha_i = \sum_{i=1}^n \beta_i = 1$.  Let
$f:C\to\mathbb{R}$ be strictly convex, where $C$ is a convex subset of a linear space.
As in Theorem \ref{AMGMrFu}, we  set  $A:=\{i: \alpha_i/\beta_i = \min\{\alpha_k/\beta_k: 1\le k\le n\}, i=1, \dots, n\}$ and   
$B=\{i: \alpha_i/\beta_i = \max\{\alpha_k/\beta_k: 1\le k\le n\},
i=1, \dots, n\}$. 
Then, we have  equality on the left hand side of
the Dragomir-Jensen inequalities 
\begin{equation}\label{refJen} 
\min_{k=1, \dots, n}\left\{\frac{\alpha_{k}}{\beta_{k}}\right\}\left(\sum_{i=1}^n \beta_i f(x_i) 
- f\left(\sum_{i=1}^n \beta_i x_i\right)
\right) 
\le
\end{equation}
\begin{equation}\label{refJen2}
\sum_{i=1}^n \alpha_i f(x_i) - f\left(\sum_{i=1}^n \alpha_i x_i\right)
\le
\max_{k=1, \dots, n}\left\{\frac{\alpha_{k}}{\beta_{k}}\right\}\left(\sum_{i=1}^n \beta_i f(x_i) 
- f\left(\sum_{i=1}^n \beta_i x_i\right)
\right)
\end{equation}
 if and only if 
for every $j\in \{1, \cdots, n\} \setminus A$, 
\begin{equation}\label{eqleftJen} 
x_j 
= 
\sum_{i\in A} \frac{\alpha_i}{\sum_{i\in A} \alpha_i} x_i
= 
 \sum_{i=1}^n \alpha_i x_i, \mbox{ \ or equivalently, \ }
x_j 
= 
\sum_{i\in A} \frac{\beta_i}{\sum_{i\in A} \beta_i} x_i
=  
 \sum_{i=1}^n \beta_i x_i,
\end{equation}
while equality holds on the right  hand side of (\ref{refJen})-(\ref{refJen2}) if and only if 
for every $j\in \{1, \cdots, n\} \setminus B$, 
\begin{equation}\label{eqrightJen} 
x_j 
= 
\sum_{i\in B} \frac{\alpha_i}{\sum_{i\in B} \alpha_i} x_i
= 
 \sum_{i=1}^n \alpha_i x_i, \mbox{ \ or equivalently, \ }
x_j 
= 
\sum_{i\in B} \frac{\beta_i}{\sum_{i\in B} \beta_i} x_i
=  
 \sum_{i=1}^n \beta_i x_i.
\end{equation}
\end{proposition}

\begin{proof} As in the proof of Theorem \ref{AMGMrFu}, we set $q_i := \alpha_i/\beta_i$ and
 suppose that $q_1\le \cdots \le q_n$. If $r:= |A| =n$, then $\alpha_i =\beta_i$ for all $i=1,\dots, n$,  
and equality trivially holds on both (\ref{refJen}) and (\ref{refJen2}),
without any conditions on  the $x_i$'s (and without needing the strict convexity of $f$). 
If $r < n$, then $r < n - s +1\le n$, and 
$q_{n-s} < q_{n-s + q} =\cdots = q_{n}$. Assume that equality holds
on (\ref{refJen2}), or, equivalently, after removing the zero summands, that
\begin{equation}\label{step1Jen} 
f\left(\sum_{i=1}^n \beta_i x_i\right)
=
\sum_{i=1}^{n-s} \left(\beta_{i} - \min_{k=1, \dots, n}\left\{\frac{\beta_{k}}{\alpha_{k}}\right\} \alpha_i
\right) f(x_i)  +  
\min_{k=1, \dots, n}\left\{\frac{\beta_{k}}{\alpha_{k}}\right\} f\left(\sum_{i=1}^n \alpha_i x_i\right).
\end{equation}
Using the equality condition for strictly convex functions in Jensen's inequality,  we see that (\ref{step1Jen}) holds if and only if
\begin{equation}\label{step2Jen}
c:= x_1 =\cdots = x_{n-s} 
 =
\sum_{i=1}^n \alpha_i x_i
= 
 \sum_{i=1}^n \beta_i x_i.
\end{equation}
Replacing in the preceding sums $x_i$ by $c$  for $1 \le i\le n-s$, we conclude that 
(\ref{step2Jen}) holds if and only if
\begin{equation}\label{step3Jen}
c
 =
\sum_{i=n-s +1}^n \frac{\alpha_i}{\sum_{i=n-s +1 }^n \alpha_i} x_i,
\end{equation}
or equivalently, if and only if
\begin{equation}\label{step4Jen}
c
 =
\sum_{i=n-s +1}^n \frac{\beta_i}{\sum_{i=n-s +1 }^n \beta_i} x_i.
\end{equation}
The proof for the left hand side of (\ref{refJen})-(\ref{refJen2}) is entirely analogous; it
can also be obtained from the right hand side.
\end{proof}

\section{Refinements of H\"older's inequality}

When the AM-GM inequality is specialized to just
 two terms, it is usually called Young's inequality.  
Next we utilize the preceding results  to generalize \cite[Lemma 2.1]{A4}
and \cite[Theorem 2.2]{A4}, giving two-sided refinements of
H\"older's inequality for  two functions.
When considering $L^p = L^p(\mu)$ spaces, 
we always assume that $\mu$ is not identically zero.

\begin{corollary}\label{betteryoungl}  Let $1 < p < \infty$ and let $q = p/(p-1)$ be its conjugate
exponent. Then for all $u,v \ge 0$ and all $\beta\in (0,1)$,
\begin{equation}\label{betteryoung}
\min \left\{\frac{1}{\beta p}, \frac{1}{(1 - \beta) q}\right\}
\left(\beta u^{p} + (1-\beta) v^{q} - u^{\beta p} 
v^{(1-\beta)q} \right) \le
\frac{u^p}{p} + \frac{v^q}{q} - uv   
\end{equation}
\begin{equation}\label{betteryoung2}
\le 
\max \left\{\frac{1}{\beta p}, \frac{1}{(1 - \beta) q}\right\}
\left(\beta u^{p} + (1-\beta) v^{q} - u^{\beta p} 
v^{(1-\beta)q} \right). 
\end{equation}
\end{corollary}

\begin{proof} Set $n=2$, $\alpha = 1/p$, $1 - \alpha = 1/q$, $x_1 = u^p$ and $x_2=v^q$ in Theorem \ref{AMGMrFu}.
\end{proof}

\begin{theorem}\label{betterhold}  Let  $1 < p < \infty$, and let $q = p/(p-1)$ be its
conjugate exponent. If $f\in L^p$, $g\in L^q$, $f, g\ge 0$, and $\|f\|_p, \|g\|_q > 0$,  then
for all $\beta\in (0,1)$, 
\begin{equation}\label{bonhold}
\|f\|_p\|g\|_q \left(1 - \max \left\{\frac{1}{\beta p}, \frac{1}{(1 - \beta) q}\right\} \left(1 - 
 \frac{\int f^{\beta p} g^{(1 - \beta) q}}{\left(\int f^{p}\right)^\beta
\left(\int g^{q}\right)^{1 - \beta}}\right) \right)  \le \|fg\|_1 
\end{equation}
\begin{equation}\label{bonhold2}
\le \|f\|_p\|g\|_q \left(1 - \min \left\{\frac{1}{\beta p}, \frac{1}{(1 - \beta) q}\right\} \left(1 - 
 \frac{\int f^{\beta p} g^{(1 - \beta) q}}{\left(\int f^{p}\right)^\beta
\left(\int g^{q}\right)^{1 - \beta}}\right) \right).
\end{equation}
\end{theorem}
\begin{proof}  The standard derivation of H\"older's inequality from
Young's inequality is applicable: Write
$u = f(x)/\|f\|_{p}$ and $v = g(x)/\|g\|_{q}$ in (\ref{betteryoung})-(\ref{betteryoung2}), integrate and reorganize terms. The rearranging of
terms is justified, since it follows from the proof
that $\int f^{\beta p} g^{(1 - \beta) q} \le 
 \left(\int f^{p}\right)^\beta
\left(\int g^{q}\right)^{1 - \beta} < \infty,
$ so all the quantities involved in (\ref{bonhold})
and (\ref{bonhold2}) are finite (alternatively, 
since $f^{\beta p}\in L^{1/\beta}$ and  $g^{(1 - \beta) q} \in L^{1/(1-\beta)}$, the bound $\int f^{\beta p} g^{(1 - \beta) q} < \infty$ can also be deduced  from H\"older's inequality).
\end{proof}

Observe that setting $\beta = 1/p$ in (\ref{bonhold})
and (\ref{bonhold2}), the preceding inequalities
become trivial equalities, as was to be expected from the equality conditions
in Theorem \ref{AMGMrFu}. When $\beta = 1/2$, the formulas in (\ref{bonhold}) and (\ref{bonhold2})
can be expressed using the angular distance $\left\|\frac{f^{p/2}}{\|f\|_p^{p/2}}-\frac{g^{q/2}}{\|g\|_q^{q/2}}\right\|_2$
between the $L^2$ functions $f^{p/2}$ and $g^{q/2}$
(cf. \cite[Theorem 2.2]{A4}) so if the
 the angular distance is ``large", then
  $\|fg\|_1$ is ``small", and 
viceversa. We are not aware of
any simple geometric interpretation of the bounds  in (\ref{bonhold})-(\ref{bonhold2}) when $\beta \neq 1/2$. 

We finish this section by stating the corresponding
refinement of H\"older's inequality for several
functions, in the simplest case $\beta_i = 1/n$
(of course, other values of $\beta_i$ can be
used if it is convenient). The proof is standard and therefore ommited. 

\begin{theorem}\label{betterhold2}  For $i = 1,\dots, n$, let $1 < p_i < \infty$ 
be such that $p_1^{-1} + \cdots + p_n^{-1} = 1$, and let $0\le f_i\in L^{p_i}$
satisfy  $\|f_i\|_{p_i}  > 0$.  Writting 
$p^{-1}_{\min} = \min\{p_1^{-1}, \dots, p_n^{-1}\}$
and $p^{-1}_{\max} = \max\{p_1^{-1}, \dots, p_n^{-1}\}$, we have
\begin{equation}\label{bonhold3} 
\prod_{i=1}^n\|f_i\|_{p_i} \left(1 - n p^{-1}_{\max}\left( 1 - \frac{\int \prod_{i=1}^n f_i^{p_i/n}}{\prod_{i=1}^n \|f_i\|_{p_i}^{p_i/n}} \right)\right) \le
\left\|\prod_{i=1}^n f_i\right\|_1 
\end{equation}
\begin{equation}\label{bonhold4}
\le  
 \prod_{i=1}^n\|f_i\|_{p_i} \left(1 - n p^{-1}_{\min}\left( 1 - \frac{\int \prod_{i=1}^n f_i^{p_i/n}}{\prod_{i=1}^n \|f_i\|_{p_i}^{p_i/n}} \right)\right).
\end{equation}
\end{theorem}

\section{Probabilistic considerations} 
By comparability of AM-GM differences,
known results about the typical behavior of
the AM-GM inequality in the equal weights case, can be used to give bounds
for other sequences of weights on the same
probability spaces. Given  $n\ge 2$  and $x = (x_1,\dots,x_n) \in \mathbb{R}^n\setminus\{0\}$, 
the
equal weights GM-AM ratio is
\begin{equation}\label{ratio0} 
r_n(x):= 
\frac{\prod_{i=1}^n |x_i|^{1/n}}{ n^{-1}\sum_{i=1}^n |x_i|}.
\end{equation}
By the AM-GM inequality we know that
$0\le r_n(x) \le 1$ always. If
 each $x_i$ is chosen independently from
$[0, \infty)$ 
according to a fixed exponential distribution, i.e.,
with probability density function $f_\lambda (t) =
\lambda e^{-\lambda t}$, then by \cite[Theorem 5.1]{Gl},   with
probability 1
\begin{equation}\label{ratio}
\lim_{n\to\infty} r_n(x) = e^{-\gamma},
\end{equation}
where $\gamma$ is Euler's constant and $e^{-\gamma}\approx 0.5615$. 
Observe that the limit does not in any way  depend on
the parameter $\lambda$. The following 
related result appears in \cite[Corollary 2.2]{A3}:
Denoting by $\mathbb{S}_{1}^{n-1}$ the $\ell_1^n$
unit sphere in $\mathbb{R}^d$, i.e.,
$\mathbb{S}_{1}^{n-1} = \{x\in \mathbb{R}^d
: |x_1| + \cdots + |x_n| = 1\}$, we have

\begin{theorem} \label{conc} 
Let  $k, \varepsilon > 0$, and let 
$P_1^{n-1}$ be the uniform
probability on $\mathbb{S}_{1}^{n-1}$. Then there exists an 
$N = 
N(k,\varepsilon)$ such that for every $n\ge N$, 
\begin{equation}\label{concentration1}
P_{1}^{n-1}\left\{(1 - \varepsilon)  e^{-\gamma}
 < r_n(x)
 < (1 + \varepsilon)  e^{-\gamma} \right\}\ge 
1 - \frac{1}{n^k}.
\end{equation}
\end{theorem}

Next we explain why the notions of random choice
in the above results are 
equivalent (as mentioned in \cite[Remark 2.4]{A3}). Observe that $r_n(x)$ is homogeneous of
degree zero, that is, constant on the
rays $t x$, where $t > 0$, $x\ne 0$. In particular,
taking $t^{-1} = \sum_{i=1}^n |x_i|$, we may
assume that $\sum_{i=1}^n |x_i| = 1$, or equivalently, that $x \in \mathbb{S}_{1}^{n-1}$. It is
then natural to define random choice by taking normalized
area on $\mathbb{S}_{1}^{n-1}$ as our probability
measure. 
Suppose next that we select points from the whole space $\mathbb{R}^n$, according to
an exponential
density $2^{-n} \lambda^{n} e^{-\lambda \|x\|_1}$ on $\mathbb{R}^{n}$ for some fixed $\lambda > 0$, or 
equivalently,  $\lambda^{n} e^{-\lambda \|x\|_1}$ on the positive cone $[0,\infty)^{n}$. Of course, an exponential distribution
gives a larger probability to ``small" vectors
than to large vectors, but this has no effect on
the result by zero homogeneity, and for the
same reason, it does not make
any difference  which $\lambda > 0$ we select. 
While all of this is intuitively obvious, for
completeness we present the formal argument. 

\begin{proposition} \label{equivofprob} 
Let 
$P_1^{n-1}$ be the uniform
probability on $\mathbb{S}_{1}^{n-1}$, and set 
$d P_n := 2^{-n} \lambda^{n} e^{-\lambda \|x\|_1} dx$, where $\lambda > 0$.
 Then for every $u\in \mathbb{R}$, 
 $P_{n}\left(\{r_n > u\}\right) = P_1^{n-1}\left(\{r_n > u\}\right)$.
\end{proposition}

Let us recall the coarea formula 
(for additional information cf. \cite{Fe}, pp. 248-250, or \cite{EG}, pp. 117-119):
\begin{equation}\label{coarea}
\int_{\mathbb{R}^{n}} g(x) |Jf(x)| dx =
\int_{\mathbb{R}} \int_{\{f^{-1}(t)\}} g (x)  d{\mathcal{H}^{n-1}(x)} d t.
\end{equation}
Here $f$ is assumed to be Lipschitz, $\mathcal{H}^{n-1}$ is the $n - 1$ dimensional
Hausdorff measure, and 
$|Jf(x)| :=\sqrt{\operatorname{det} df(x) df(x)^t}$ denotes the
modulus of the Jacobian. 

\begin{proof} It is well known and easy to check that
the volume of the  $(\mathbb{R}^n, \|\cdot\|_1)$-unit ball is
$|\mathbb{B}_1^{n}| = 2^{n}/n!$. The area 
$|\mathbb{S}_{1}^{n-1}| = 
\mathcal{H}^{n-1}(\mathbb{S}_{1}^{n-1})$
of
the $\ell_1$ unit sphere then follows from
the coarea formula: 
For every
$x\in \mathbb{R}^{n}\setminus \cup_{i = 1}^n \{x_i\ne 0\}$,
the function
$f(x)= \|x\|_{1}$ is differentiable, and
$$
|Jf(x)|=\sqrt{\operatorname{det} (df(x) df(x)^t)}= \sqrt{\sum_{i=1}^n 1} = \sqrt{n}
$$ a.e. on
$\mathbb{R}^{n}$.  
Set
$g(x) = \chi_{\mathbb{B}_1^{n}}/|Jf(x)|$ in (\ref{coarea}), and denote
by $\mathbb{S}^{n-1}_{1}(\rho)$ the sphere centered at 0 of
radius $\rho$ (when $\rho = 1$ we  omit it). Then 
\begin{equation*}
\frac{ 2^{n}}{n!} 
= 
|\mathbb{B}^{n}_{1}| 
= \int_{\mathbb{B}^{n}_{1}}   d x
=
\int_{0}^{1} \int_{\mathbb{S}^{n-1}_{1}(\rho)} \frac{1}{\sqrt{n}} d \mathcal{H}^{n-1}(x) d \rho 
=
\frac{|\mathbb{S}^{n-1}_{1}|}{\sqrt{n}}\int_{0}^{1} \rho^{n-1} d \rho
 =   
 \frac{|\mathbb{S}_{1}^{n-1}|}{ n \sqrt{n}},
\end{equation*}
so $|\mathbb{S}_{1}^{n-1}| = 2^n \sqrt{n} /\Gamma(n)$. Observe that 
$ P_1^{n-1}\left(A\right)
 = 
\mathcal{H}^{n-1}(A\cap \mathbb{S}_{1}^{n-1})/|\mathbb{S}_{1}^{n-1}|$.

Recalling that the ratio $r_n$ is homogeneous of
degree 0,
so $r_n(x) = r_n(x/\|x\|_1)$, we
next  set
$g(x) = \chi_{\{r_n > u\}}(x)  \exp\left(  - \lambda \sum_{i=1}^n  |x_i| \right) /|Jf(x)|$ in (\ref{coarea}). Since $|Jf(x)| = \sqrt n$ a.e.,
\begin{equation*}
P_{n}\left(\{r_n > u\}\right)
=
\frac{\lambda^n}{2^n} \int_{\mathbb{R}^n} 
\chi_{\{r_n > u\}}(x)  \exp\left(  - \lambda \sum_{i=1}^n  |x_i| \right)  dx
\end{equation*}
\begin{equation*}
=
 \frac{\lambda^n}{2^n\sqrt{n}} \int_0^\infty  \int_{\{\|x\|_{1} = t\}} \chi_{\{r_n > u\}}(x) 
 e^{- \lambda t} d\mathcal{H}^{n-1}(x) dt
\end{equation*}
\begin{equation*}
=  \frac{\lambda^n}{2^n\sqrt{n}} \int_0^\infty e^{- \lambda t} t^{n-1} \int_{\{\|x\|_{1} = 1\}} \chi_{\{r_n > u\}}(x) 
 d\mathcal{H}^{n-1}(x) dt
\end{equation*}
\begin{equation*}
 =  \frac{\Gamma(n)}{2^n\sqrt{n}} 
 \int_{\mathbb{S}_{1}^{n-1}} \chi_{\{r_n > u\}}(x) 
d\mathcal{H}^{n-1}(x)   
 = P_1^{n-1}\left(\{r_n > u\}\right).
 \end{equation*}
\end{proof}

The next result is stated in terms of independent
choices from an exponential distribution on
$[0,\infty)$ rather than on $\mathbb R$ (so we can write $x_i$ instead of
$|x_i|$, and thus $\|x\|_1 = \sum_{i=1}^n x_i$). As usual, the weights $\alpha_{i,n} > 0$ 
are assumed to satisfy 
$\sum_{i=1}^n \alpha_{i,n}= 1$, and the largest and smallest such weights are denoted by
$\alpha_{\max,n}$ and $\alpha_{\min,n}$ respectively. We suppose that for each $n\ge 2$ we are
given a sequence of weights $\{\alpha_{1,n}, \dots, \alpha_{n,n}\}$.

\begin{theorem} \label{concbounds} 
Let  $k, \lambda, \varepsilon > 0$, and for 
$1 \le i \le n$, let $x_i\in [0,\infty)$ be chosen
independently
according to an exponential distribution with
parameter $\lambda$. Let $P_n$ denote the
product probability on $[0,\infty)^n$ with
density $\lambda^{n} e^{-\lambda \|x\|_1}$.
 Then there exists an 
$N = 
N(k,\varepsilon)$ such that for every $n\ge N$, 
\begin{equation}\label{concentration11}
P_{n}\left\{\left[1 - (1 + \varepsilon)  e^{-\gamma}\right]
\alpha_{\min,n} \|x\|_1 
 < \sum_{i=1}^n \alpha_{i,n} x_i - \prod_{i=1}^n x_i^{\alpha_{i,n}}
 < \left[1 - (1 - \varepsilon)  e^{-\gamma}\right]
\alpha_{\max,n} \|x\|_1  \right\}
\end{equation}
\begin{equation}\label{concentration2}
\ge 
1 - \frac{1}{n^k}.
\end{equation}
\end{theorem}
 
\begin{proof} Since
\begin{equation*}
\left\{x\in [0,\infty)^n : (1 - \varepsilon)  e^{-\gamma}
 < r_n (x) 
 < (1 + \varepsilon)  e^{-\gamma}\right\}
\end{equation*}
\begin{equation*}
=
\left\{x\in [0,\infty)^n : \left[1 - (1 + \varepsilon)  e^{-\gamma}\right]
\frac{1}{n} \|x\|_1 
 < \frac{1}{n} \sum_{i=1}^n x_i - \prod_{i=1}^n x_i^{1/n}
 < \left[1 - (1 - \varepsilon)  e^{-\gamma}\right]
\frac{1}{n} \|x\|_1  \right\},
\end{equation*}
the result follows from Corollary
\ref{cor1} together with Theorem \ref{conc} (expressed
in terms of an exponential distribution rather
than normalized surface area on $\mathbb{S}_{1}^{n-1}$).
\end{proof}

A result analogous to the previous one can be stated for
the GM-AM ratio, using the following bounds due to
S. S. Dragomir, cf. \cite[Section 4]{Dra}:
\begin{equation}\label{ratioDra}
r_n(x)^{n \alpha_{\max,n}} \le 
\frac{\prod_{i=1}^n x_{i}^{\alpha_{i,n}}}{\sum_{i=1}^n \alpha_{i,n} x_{i}}
\le 
r_n(x)^{n \alpha_{\min,n}}.
\end{equation}

\begin{theorem} \label{concboundsratio} 
Let  $k, \lambda, \varepsilon > 0$, and for 
$1 \le i \le n$, let $x_i\in [0,\infty)$ be chosen
independently
according to an exponential distribution with
parameter $\lambda$. Let $P_n$ denote the
product probability on $[0,\infty)^n$ with
density $\lambda^{n} e^{-\lambda \|x\|_1}$.
Then there exists an 
$N = 
N(k,\varepsilon)$ such that for every $n\ge N$, 
\begin{equation}\label{concentrationratio}
P_{n}\left\{\left(1 -  \varepsilon\right) e^{- n \alpha_{\max,n} \gamma}
 < \frac{\prod_{i=1}^n x_{i,n}^{\alpha_i}}{\sum_{i=1}^n \alpha_{i,n} x_i}
 < \left(1 + \varepsilon\right)  e^{-n \alpha_{\min,n} \gamma}\right\}
\end{equation}
\begin{equation}\label{concentrationratio2}
\ge 
1 - \frac{1}{n^k}.
\end{equation}
\end{theorem}
 
\begin{proof} This easily follows from (\ref{ratioDra})
 together with Theorem \ref{conc}, expressed
in terms of an exponential distribution 
instead of normalized surface area on $\mathbb{S}_{1}^{n-1}$.
\end{proof}

Additional probabilistic
results regarding the GM-AM ratio for sequences
of unequal weights can be found in \cite{A3}. The main difference
between these results and Theorems  \ref{concbounds}-\ref{concboundsratio}, 
is that in \cite{A3} the probability distributions are chosen depending
on the sequences of weights, while above,
the same density $\lambda^{n} e^{-\lambda \|x\|_1}$ is used for all sequences of 
$n$ weights $\alpha_{i,n}$.

\end{document}